\documentclass[bibalpha]{amsart}
\usepackage{verbatim}
\usepackage{enumerate}

\title{dp-minimal valued fields}

\author{Franziska Jahnke} 
\address{Westf\"alische Wilhelms-Universit\"at M\"unster\\Institut f\"ur Mathematische Logik und Grundlagenforschung\\Einsteinstr. 62\\48149 M\"unster, 
Germany}
\email{franziska.jahnke@uni-muenster.de}

\author{Pierre Simon}
\address{Universit\'e Claude Bernard - Lyon 1\\
Institut Camille Jordan\\
43 boulevard du 11 novembre 1918\\
69622 Villeurbanne Cedex, France}
\email{simon@math.univ-lyon1.fr}
\thanks{Partially supported by ValCoMo (ANR-13-BS01-0006).}

\author{Erik Walsberg}
\address{Department of Mathematics\\
University of California, Los Angeles\\
Box 951555\\
Los Angeles, CA 90095-1555,
USA}
\email{erikw@math.ucla.edu}

\usepackage{enumerate}
\usepackage{braket}
\usepackage{txfonts}
\newtheorem{Th}{Theorem}[section]
\newtheorem{Thm}[Th]{Theorem}

\newtheorem*{Def}{Definition}
\newtheorem{Cor}[Th]{Corollary}
\newtheorem{Prop}[Th]{Proposition}

\newtheorem{Lem}[Th]{Lemma}

\newtheorem*{Lem*}{Lemma}

\begin{document}
\begin{abstract}We show that dp-minimal valued fields are henselian and that a dp-minimal field admitting a definable type V topology is either real closed, algebraically closed or admits a non-trivial definable henselian valuation.
We give classifications of 
dp-minimal ordered abelian groups and dp-minimal ordered fields without additional structure.
\end{abstract}

\maketitle

\section{Introduction}

Very little is known about NIP fields. 
It is widely believed that an NIP field is either real closed, separably closed or admits a definable henselian valuation. Note that even the stable case of this conjecture is open.

In this paper we study a very special case of this question: that of a valued or ordered dp-minimal field. Dp-minimality is a combinatorial generalization of o-minimality and C-minimality. A dp-minimal structure can be thought of as a \emph{one-dimensional} NIP structure. The main result of this paper is that a dp-minimal valued field is henselian. As a consequence, we show that an $\aleph_1$-saturated dp-minimal ordered field admits a henselian valuation with residue field $\mathbb{R}$. These results can be seen either as a special case of the conjecture on NIP fields or as a generalization of what is known in the C-minimal and weakly o-minimal cases.

Our proof has two parts.
We first establish some facts about dp-minimal topological structures.
In Section~\ref{section:uniform} we generalize statements and proofs of Goodrick \cite{Goodrick} and Simon \cite{Simon:dp-min} on ordered dp-minimal structures to the more general setting of a dp-minimal structure admitting a definable \emph{uniform topology}. This seems to be the most general framework to which the proofs apply. Only afterwords will we assume that the topology comes either from a valuation or an order.
Once we have established the necessary facts about dp-minimal uniform structures the remainder of the proof, given in Section~\ref{section:dpminfields}, follows part of the proof that weakly o-minimal fields are real closed.
The proof that weakly o-minimal fields are real closed is surprisingly more complicated then the proof that o-minimal fields are real closed.
It involves first finding some henselian valuation and then showing that the residue field is real closed and the value group is divisible, from which the result follows. This argument was extended by Guingona \cite{G14} to dp-small structures, a strengthening of dp-minimality. We follow again the same path, but assuming only dp-minimality, the value group is not necessarily divisible.

In Section~\ref{section:dpmingroup} we use the Gurevich-Schmitt quantifier elimination for ordered abelian groups to show that an ordered abelian group $\Gamma$ without additional structure is dp-minimal if and only if it is is \emph{non-singular}, that is if $|\Gamma/p\Gamma| < \infty$ for all primes $p$.
It follows that an $\aleph_1$-saturated dp-minimal ordered field admits a henselian valuation with residue field $\mathbb{R}$ and non-singular value group.
 In fact, this is best possible: Chernikov and Simon \cite{CS:inp-min} show that any such field is dp-minimal.

In Section~\ref{section:typeV} we show that a dp-minimal valued field is either real closed, algebraically closed or admits a henselian valuation which is definable in the language of rings.
This follows easily from results in valuation theory and the theory of definable valuations on henselian fields developed in \cite{JK14a} and \cite{JK14}.
We show that a dp-minimal field admitting a definable type V topology has an externally definable valuation, it follows that a dp-minimal field admitting a definable type V topology is henselian and thus admits a definable henselian valuation.
In the final part of the section, for any dp-minimal ordered field which is not real closed we give an explicit fomula in the language $\mathcal{L}_\textrm{ring}\cup\{\leq\}$ defining a non-trivial henselian valuation.
%In the speculative Section~\ref{section:spec} we show that if every infinite dp-minimal field is either real closed, algebraically closed or %admits a definable valuation then every dp-minimal field is either real closed, algebraically closed or admits  a henselian valuation with %real closed, algebraically closed or finite residue field.
In Proposition~\ref{prop1} we show that an ordered field which is not dense in its real closure admits a definable convex valuation.
This result may be of independent interest.

After finishing this paper, we learnt from Will Johnson that he simultaneously proved more general results on dp-minimal fields using
different methods. In fact, he gives a complete classification. Most of our results can be deduced from his work. 

\section{dp-minimality}
Let $L$ be a multisorted language with a distinguished home sort and let $T$ be a theory in $L$.
Throughout this section $\mathcal M$ is an $|L|^+$-saturated model of $T$ with distinguished home sort $M$.
We recall the definitions of inp- and dp-minimality.
These definitions are usually stated for one-sorted structures, we define them in our multisorted setting.

\begin{Def}
The theory $T$ is \textbf{not inp-minimal} with respect to the home sort if there are two formulas $\phi(x;\bar y)$ and $\psi(x;\bar z)$, where $x$ is a single variable of sort $M$, two sequences $(\bar b_i:i<\omega)$ and $(\bar c_i:i<\omega)$ in $\mathcal M$ and $k \in \mathbb{N}$ such that:
\begin{itemize}
\item the sets $\{\phi(x;\bar b_i):i<\omega\}$ and $\{\psi(x;\bar c_i):i<\omega\}$ are each $k$-inconsistent;
\item $\phi(x;\bar b_i)\wedge \psi(x;\bar c_j)$ is consistent for all $i,j < \omega$.
\end{itemize}
\end{Def}
If $T$ is inp-minimal and one-sorted then $T$ is NTP$_2$.
\begin{Def}
The theory $T$ is \textbf{not dp-minimal} with respect to the home sort if there are two formulas $\phi(x;\bar y)$ and $\psi(x;\bar z)$, where $x$ is a single variable of sort $M$, two sequences $(\bar b_i:i<\omega)$ and $(\bar c_i:i<\omega)$ in $\mathcal M$ such that:
\begin{itemize}
\item for any $i_*,j_* \in \omega$, the conjunction $$\phi(x;\bar b_{i_*})\wedge \left[\bigwedge_{i\neq i_*} \neg \phi(x;\bar b_i)\right] \wedge \psi(x;\bar c_{j_*})\wedge \left[\bigwedge_{j\neq j_*} \neg \psi(x;\bar c_j)\right]$$
is consistent.
\end{itemize}
\end{Def}

If $T$ is one-sorted then $T$ is dp-minimal if and only if $T$ is both inp-minimal and NIP (see for example \cite{Simon:dp-min}).
It follows that if $T$ is $NIP$ and inp-minimal with respect to the home sort then $T$ is dp-minimal with respect to the home sort.

We define the \textbf{Shelah expansion} $\mathcal M^{Sh}$ of $\mathcal M$.
Fix an $|\mathcal M|^+$-saturated elementary expansion $\mathcal M\prec \mathcal N$.
We say that a subset of $M$ is \textbf{externally definable} if it is of the form $\{ a \in M : \mathcal N \models \phi(a;b)\}$ for some $L$-formula $\phi(x;y)$ and $b \in \mathcal N^{|y|}$.
A straightforward application of $|\mathcal M|^+$-saturation shows that the collection of externally definable sets does not depend on the choice of $\mathcal N$.
Note that the externally definable sets form a boolean algebra.
The Shelah expansion $\mathcal M^{Sh}$ is the expansion of $\mathcal M$ by predicates for all externally definable subsets of $M$.
More precisely, for each formula $\phi(x;y)\in L$ and each $b\in \mathcal N^{|y|}$, we have a relation $R_{\phi(x;b)}(x)$ which is interpreted as $\{a \in M : \mathcal N\models \phi(a;b)\}$.
 Shelah proved the following, see e.g. \cite[Chapter 3]{Simon:book}.

\begin{Prop}[Shelah]
If $\mathcal M$ is NIP then $\mathcal M^{Sh}$ admits elimination of quantifiers and is consequently also NIP.
\end{Prop}

We make extensive use of the next fact which was observed in \cite{OnUs1}.

\begin{Prop}\label{prop_shelah}
If $T$ is $NIP$ and dp-minimal with respect to the home sort then $\mathrm{Th}(\mathcal M^{Sh})$ 
is also dp-minimal (relative to the same home sort).
\end{Prop}
\begin{proof}
Since $T$ is dp-minimal, it is NIP, hence the structure $\mathcal M^{Sh}$ 
has elimination of quantifiers. We suppose towards a contradiction that 
$\mathrm{Th}(\mathcal M^{SH})$ is not dp-minimal. Then there are two formulas 
$\phi(x;\bar y)$, $\psi(x;\bar z)$ such that for any $n<\omega$, there are 
finite sequences $(a_i:i<n)$ and $(b_i:i<n)$ of elements of $M$ with the 
property as in the definition of dp-minimality. 
As $\mathcal M^{SH}$ admits quantifier elimination there are tuples 
$d,d'\in \mathcal N$ and $L$-formulas $\phi'(x;\bar y;d),\psi'(x;\bar z;d')$ 
such that for any $(a,b)\in \mathcal M^{|x|}\times \mathcal M^{|y|}$, we have 
$\mathcal M^{Sh}\models \phi(a;b) \iff \mathcal N\models \phi'(a;b;d)$ 
likewise for $\psi$ and $\psi'$. 
Using compactness, we see that the formulas $\phi'(x;\bar y;d)$ and $\psi'(x;\bar z;d')$ contradict the dp-minimality of $T$.
\end{proof}

If $M$ admits a definable linear order then every convex subset of $M$ is externally definable, so the expansion of an ordered dp-minimal structure by a convex set is always dp-minimal.
We end the section with a proof of the well-known fact that dp-minimal fields are perfect.
The proof shows that a field of finite dp-rank is perfect.

\begin{Lem}\label{perfect}
Every dp-minimal field is perfect.
\end{Lem}

\begin{proof}
Let $K$ be a nonperfect field of characteristic $p>0$.
We let $K^p$ be the set of $p$th powers.
Fix $z \in K \setminus K^p$.
Let $\tau : K \times K \rightarrow K$ be given by $\tau(x,y) = x^p + zy^p$.
We show that $\tau$ is injective, it follows that $K$ is not dp-minimal.
Fix $x_0,x_1,y_0,y_1 \in K$ and suppose that $\tau(x_0,y_0) = \tau(x_1,y_1)$ that is:
$$ x_0^p + zy_0^p = x_1^p + zy_1^p.$$
If $y_0 \neq y_1$ then $y_0^p \neq y_1^p$, so:
$$ z = \frac{ x_0^p - x_1^p }{ y_0^p - y_1^p } \in K^p, $$
so we must have $y_0 = y_1$.
Then $ x_0^p = x_1^p$ and so also $x_0 = x_1$.
Thus $\tau$ is injective.
\end{proof}
\section{Dp-minimal uniform spaces}\label{section:uniform}
Many dp-minimal structures of interest admit natural ``definable topologies''.
Typically the topology is given by a definable linear order or a definable valuation.
In this section we develop a framework designed to encompass many of these examples, dp-minimal structures equipped with definable uniform structures.
We prove some results of Goodrick \cite{Goodrick} and Simon \cite{Simon:dp-min} in this setting.
The proofs are essentially the same.
We first recall the notion of a uniform structure on $M$.
We let $\Delta$ be the diagonal $\{ (x,x) : x \in M \}$.
Given $U,V \subseteq M^2$ we let 
$$U \circ V = \{ (x,z) \in M^2 : (\exists y \in M) (x,y) \in U, (y,z) \in V\}.$$
A \textbf{basis} for a uniform structure on $M$ is a collection $\mathcal B$ of subsets of $M^2$ satisfying the following:
\begin{enumerate}
\item the intersection of the elements of $\mathcal B$ is equal to $\Delta$;
\item if $U \in \mathcal B$ and $(x,y) \in U$ then $(y,x) \in U$;
\item for all $U,V \in \mathcal B$ there is a $W\in \mathcal B$ such that $W\subseteq U\cap V$;
\item for all $U\in \mathcal B$ there is a $V\in \mathcal B$ such that $V\circ V \subseteq U$.
\end{enumerate}
The \textbf{uniform structure} on $M$ generated by $\mathcal B$ is defined as 
$$\tilde{\mathcal B} = \{U \subseteq M^2 : 
(\exists V\in \mathcal B)\,V\subseteq U\}.$$
 Elements of $\tilde{\mathcal B}$ are called \textbf{entourages} and elements of $\mathcal B$ are called \textbf{basic entourages}.
Given $U \in \mathcal B$ and $x \in M$ we let $U[x] = \{y : (x,y)\in U\}$.
 As usual, one defines a topology on $M$ by declaring that a subset $A\subseteq M$ is open if for every $x\in A$ there is $U\in  \mathcal B$ such that $U[x] \subseteq A$.
 The first condition above ensures that this topology is Hausdorff.
 The collection $\{ U[x] : U \in \mathcal B\}$ forms a basis of neighborhoods of $x$. 
We will refer to them as \textbf{balls} with center $x$.
Abusing terminology, we say that $\mathcal B$ is a \textbf{definable uniform structure} if there is a formula $\varphi(x,y,\bar{z})$ such that
$$ \mathcal  B = \{ \varphi(M^2,\bar{c})\;| \; \bar{c} \in D\}$$
for some definable set $D$.
Note that the conditions above are first order conditions on $\varphi$.
We give some examples of definable uniform structures.
\begin{enumerate}
\item Suppose that $\Gamma$ is an $\mathcal M$-definable ordered abelian group and $d$ is a definable $\Gamma$-valued metric on $M$.
We  can take $\mathcal B$ to be the collection of sets of the form $\{ (x,y) \in X^2 : d(x,y) < t \}$ for $t \in \Gamma$.
The typical case is when $\Gamma = M$ and $d(x,y) = | x - y|$.
\item Suppose that $\Gamma$ is a definable linear order with minimal element and that $d$ is a definable $\Gamma$-valued ultrametric on $M$. Then we can put a definable uniform structure on $M$ in the same way as above.
The typical case is when $M$ is a valued field.
\item Suppose that $M$ expands a group. Let $D$ be a definable set and suppose that $\{ U_{\bar{z}}: \bar{z} \in D\}$ is a definable family of subsets of $M$ which forms a neighborhood basis at the identity for the topology on $M$ under which $M$ is a topological group. Then sets of the form $\{ (x,y) \in M^2 : x^{-1}y \in U_{\bar{z}}\}$ for $\bar{z} \in D$ give a definable uniform structure on $M$.
\end{enumerate}
\textit{For the remainder of this section we assume that $\mathcal B$ is a definable uniform structure on $M$ and that $\mathcal M$ is inp-minimal with respect to the home sort. We further suppose that $M$ does not have any isolated points}.

\begin{Lem}\label{lem_dense}
Every infinite definable subset of $M$ is dense in some ball.
\end{Lem}
\begin{proof}
We suppose towards a contradiction that $X \subseteq M$ is infinite, definable and nowhere dense.
 Let $(a_i:i<\omega)$ be a sequence of distinct points in $X$.
Applying $\aleph_1$-saturation we let $U_0\in \mathcal B$ be such that $U_0[a_i] \cap U_0[a_j] = \emptyset$ when $i \neq j$.
Now inductively construct elements $U_n \in \mathcal B$ and 
$x_{i,n}\in U_0[a_i]$ such that $a_i \in U_n[x_{i,n}]$ and $U_{n+1}[x_{i,n}] \cap X = \emptyset$ for all $i,n < \omega$.
Assume we have defined $U_m$ and $x_{i,m}$ for $m\leqslant n$. 
By assumption, $X$ is not dense in the ball $U_n[a_i]$ for any $i$, 
hence we can find some $U_{n+1}\in \mathcal B$ and points 
$x_{i,n}\in U_n[a_i]$ such that 
$U_{n+1}[x_{i,n}] \cap X = \emptyset$ for all $i$. 
If $U_n[x]=\pi(x,\bar{b}_n)$, then
the formulas
$$\phi(x, a_i)\equiv x \in U_0[a_i] \textrm{ and } \psi(x,\bar{b}_n, \bar{b}_{n+1})\equiv (U_{n+1}[x] \cap X = \emptyset \wedge 
U_{n}[x] \cap X \neq \emptyset)$$
witness that $T$ is not inp-minimal with respect to the home sort.
This contradiction shows that the lemma holds.
\end{proof}

We leave the simple topological proof of the following to the reader:

\begin{Cor}
Any definable closed subset of $M$ is the union of an open set and finitely many points.
Moreover, the closure of a definable open set is equal to 
itself plus finitely many points.
\end{Cor}

The next lemma is a version of Lemma 3.19 of \cite{Goodrick}.
Given $R \subseteq M^2$ and $a \in M$ we define 
$R(a) = \{b \in M : (a,b) \in R\}$.
\begin{Lem}\label{lem_positivefun}
Let $R \subseteq M^2$ be a definable relation that for every $a\in M$, there is $V\in \mathcal B$ satisfying $V[a]\subseteq R(a)$. Then there are $U,V\in \mathcal B$ such that $V[b]\subseteq R(b)$ for every $b \in U[a]$.
\end{Lem}
\begin{proof}
We suppose otherwise towards a contradiction.
We suppose that for every $a \in M$ and $U,V \in \mathcal B$ there is a $b \in U[a]$ such that $V[b] \subsetneq R(b)$.
 Let $(a_i :i<\omega)$ be a sequence of pairwise distinct elements of $M$ and fix $U\in \mathcal B$ such that the balls $U[a_i]$ are pairwise disjoint.
 For each $i<\omega$, pick some $x_{i,0} \in U[a_i]$.
 Then choose $U_1 \in \mathcal B$ such that $U_1[x_{i,0}]\subseteq R(x_{i,0})$ for all $i$.
 Next pick points $x_{i,1}\in U[a_i]$ such that $U_1[x_{i,1}] \subsetneq R(x_{i,1})$ and choose $U_2\subset U_1$ such that $U_2[x_{i,1}] \subseteq R(x_{i,1})$ holds for all $i$.
 Iterating this, we obtain a decreasing sequence $(U_k : 1\leq k<\omega)$ of elements of $\mathcal B$ such that for each $i,k<\omega$, there is a $x_{i,k}\in U[a_i]$ such that $U_{k+1}[x_{i,k}]\subseteq R(x_{i,k})$ but $U_k[x_{i,k}] \subsetneq R(x_{i,k})$. Then the formulas 
$$x\in U[a_i] \textrm{ and } 
U_{k+1}[x]\subseteq R(x)\wedge U_k[x] \subsetneq R(x)$$ 
form an ict-pattern of size 2 and contradict inp-minimality.
\end{proof}

\begin{Lem}\label{lem_nondivide}
Let $X$ be a definable set and $a\in M$.
Suppose that $X \cap U[a]$ is infinite for every $U \in \mathcal B$.
Then $X$ does not divide over $a$.
\end{Lem}
\begin{proof}
Let $\bar b$ be the parameters defining $X$ and let $\phi$ be a formula such that $X = \phi(M,\bar{b})$.
Let $X_{\bar{c}} = \phi(M, \bar{c})$.
 Assume that $X$ divides over $a$, so
there is an $a$-indiscernible sequence $\bar b = \bar b_0, \bar b_1,\ldots$ such that the intersection $\bigcap_{i<\omega} X_{\bar b_i}$ is empty. We build by induction a decreasing sequence $(U_j :j<\omega)$ of elements of $\mathcal B$ such that for any $i,j<\omega$, the intersection $(U_j[a]\setminus U_{j+1}[a])\cap X_{\bar b_i}$ is non-empty.
Suppose we have $U_0,\ldots U_n$.
For each $i < \omega$ there is a point in $U_n[a] \cap X_{\bar b_i}$ other then $a$.
Thus for all $i < \omega$ there is a $V \in \mathcal B$ such that $U_n[a]\setminus V[a]$ intersects $X_{\bar b_i}$.
An application of $\aleph_1$-saturation gives a $U_{n+1} \in \mathcal B$ such that $U_{n+1} \subseteq U_n$ and $U_{n}[a]\setminus U_{n+1}[a]$ intersects $X_{\bar b_i}$ for all $i < \omega$. 
We obtain an ict-pattern of size 2 by considering the formulas $x \in X_{\bar b_i}$ and $x\in U_j[a]\setminus U_{j+1}[a]$. This contradicts inp-minimality.
\end{proof}
\textit{For the remainder of this section we assume that $\mathcal{M}$ is NIP and hence dp-minimal.}
We say that $X, Y \subseteq M$ have the same germ at $a \in M$ if there is a $U\in \mathcal  B$ such that $U[a]\cap X = U[a]\cap Y$.

\begin{Lem}\label{lem_finitegerm}
Let $\phi(x;\bar y)$ be a formula, $x$ of sort $M$, and $a\in M$. There is a finite family $(\bar b_i :i <n)$ of parameters such that for any $\bar b\in M^{|\bar y|}$ there is an $i<n$ such that $\phi(M;\bar b)$ and $\phi(M; \bar b_i)$ have the same germ at $a$.
\end{Lem}
\begin{proof}
Let $M_0$ be a small submodel of $\mathcal M$ containing $a$. Let $\bar b_0$ and $\bar b_1$ have the same type over $M_0$. In NIP theories, non-forking, non-dividing and non-splitting all coincide over models. The formula $\phi(x;\bar b_0)\triangle \phi(x;\bar b_1)$ splits over $M_0$, therefore it divides over $M_0$. Lemma~\ref{lem_nondivide} shows that
 $$U[a]\cap (\phi(x;\bar b_0)\triangle \phi(x;\bar b_1)) = \emptyset \quad \text{for some } U \in \mathcal B.$$ This means that $\phi(x;\bar b_0)$ and $\phi(x;\bar b_1)$ have the same germ at $a$. The lemma follows by $|L|^+$-saturation.
\end{proof}

We now assume that $M$ admits a definable abelian group operation.
We assume that the group operations are continuous and that the basic entourages are invariant under the group action, i.e. $U[0] + a = U[a]$ for all $U \in \mathcal B$ and $a \in M$.
We make the second assumption without loss of generality.
If we we only assume that the group operations are continuous then we can define an invariant uniform structure whose entourages are of the form 
$$\{ (x,y) \in M\times M : x-y \in U[0]\} \quad \text{for } U \in \mathcal B.$$
We also assume that $M$ is divisible and, more precisely, assume that for every $U\in \mathcal B$ and $n$ there is a $V \in \mathcal B$ such that for all $y \in V[0]$ there is an $x \in U[0]$ such that $nx = y$.
This assumption will hold if $M$ is a field, $+$ is the field addition, and the uniform structure on $M$ is given by a definable order or valuation.
Under these assumptions we prove:

\begin{Prop}\label{prop_interior}
Every infinite definable subset of $M$ has nonempty interior.
\end{Prop}
\begin{proof}
Let $X$ be an infinite definable subset of $M$.
Consider the family of translates $$\{ X - b : b\in M\}.$$ By Lemma \ref{lem_finitegerm}, there are finitely many elements $(b_i:i<n)$ such that for any $b\in M$ there is an $i < n$ such that $X-b$ and $X-b_i$ have the same germ at $0$.
For each $i < n$ let $X_i$ be the set of $b \in X$ such that $X - b_i$ and $X-b$ have the same germ at 0. 
Fix $i < n$ such that $X_i$ is infinite.
Let $b \in X_i$ and let $V \in \mathcal B$ be such that $X- b$ and $X - b_i$ agree on $V[0]$.
It is easy to see that if $c \in V[0] \cap (X_i - b)$ then $c \in X_i - b_i$.
Likewise, if $c \in V[0]$ and $c \in X_i - b_i$ then $c \in X_i - b$.
Thus, $X_i - b$ has the same germ at 0 as $X_i - b_i$ for all $b \in X_i$.
Replacing $X$ by $X_i$, we may assume $X-b$ and $X-b'$ have the same germ at 0 for all $b,b' \in X$.

By Lemma \ref{lem_dense}, $X$ is dense in some ball.
It follows that $X$ has no isolated points.
 Translating $X$, we may assume that $0\in X$.
 For any $b\in X$, there is a $U \in \mathcal B$ such that $X$ and $X-b$ coincide on $U[0]$, equivalently $X$ and $X+b$ coincide on $U[b]$.
 Let $R$ be the relation given by
 $$R(x,y) := (y\in X \leftrightarrow y\in X+x).$$ 
We have shown that for each $b \in X$ there is a $U \in B$ such that $U[b] \subseteq R(b,M)$.
For the moment take $X$ to be the distinguished home sort of $\mathcal M$.
As $X$ has no isolated points and has dp-rank $1$ we can apply Lemma \ref{lem_positivefun} to get a $V\in \mathcal B$ and an open $W \subseteq M$ such that $X$ is dense in $W$ and $V[x]\subseteq R(x,M)$ for all $x \in W \cap X$.
Translating again, we may assume there is a $U \in \mathcal B$ such that $X$ is dense in $U[0]$ and $V[x]\subseteq R(x,M)$ for every $x\in U[0] \cap X$.
Finally, we may replace both $U$ and $V$ by some $U[0] \subseteq U\cap V$. Hence to summarize, we have the following assumption on $X$:
$$ \boxtimes \quad X \text{ is dense in } U[0] \text{ and } X \text{ and } X-b \text{ coincide on } U[0] \text{ for any } b \in X \cap U[0].$$
Pick $V\in B$ such that $V[0]-V[0]\subseteq U[0]$.\\
\emph{Claim}: If $g,h \in X\cap V[0]$, then $g-h$ is in $X$.\\
\emph{Proof of claim}: As $-h \in (X-h)\cap U[0]$, by $\boxtimes$, we also have $-h\in X\cap U[0]$. Then from $g\in X\cap U[0]$, we deduce $g\in X+h \cap U[0]$ and hence $g-h\in X$ as required.\\
Suppose that the family $\{X-b :b \in M\}$ has strictly less than $n$ distinct germs at 0. Fix $W \in \mathcal B$ such that the sum of any $n!$ elements from $W[0]$ falls in $V[0]$.\\
\emph{Claim}: For any $g\in W[0]$ we have $k\cdot g \in X$ for some $k\leq n$.\\
\emph{Proof of claim}: By Lemma \ref{lem_finitegerm} and choice of $n$, there are distinct $k,k'< n$ such that $X-kg$ and $X-k'g$ have the same germ at 0.
We suppose that $k < k'$.
 As $X$ is dense in $U[0]$, there is an $h \in V[0]$ such that $h \in X - kg$ and $h\in X - k'g$ and $kg +h, k'g + h \in V[0]$.
As $k'g+h, kg + h \in X \cap V[0]$ we have $(k'-k)g \in X$. This proves the claim.\\
Applying the assumptions on $M$ we let $W' \in \mathcal B$ be such that $W'[0] \subseteq W[0]$ and if $y \in W'[0]$ then there is an $x \in W[0]$ such that $n! \cdot x = y$.
Now pick some $g\in W'[0]$. Let $h\in W[0]$ such that $n!\cdot h = g$. For some $k\leq n$, $k\cdot h \in X$. But then $(n!/k)(kh)=g \in X$, using the first claim. Hence $W'[0]\subseteq X$.
Thus $X$ has nonempty interior.
\end{proof}

We leave the easy topological proof of the following corollary to the reader:

\begin{Cor}
Any definable subset of $M$ is the union of a definable open set and finitely many points.
\end{Cor}

For this one can show:

\begin{Cor}\label{cor:interior}
A definable subset of $M^n$ has dp-rank $n$ if and only if it has nonempty interior.
\end{Cor}

\begin{proof}
Proposition~\ref{prop_interior} above shows that the Corollary holds in the case $n = 1$.
Proposition 3.6 of \cite{Simon:dp-rank} shows that the Corollary holds if $M$ expands a dense linear order and every definable subset of $M$ is a union of an open set and finitely many points. 
The proof of Proposition 3.6 of \cite{Simon:dp-rank} goes through \textit{mutatis mutandis} in our setting.
\end{proof}

The following Proposition is crucial in the proof that a dp-minimal valued field is henselian.

\begin{Prop}\label{prop_intpreservation}
Let $X\subseteq M^n$ be a definable set with non-empty interior. Let $f:M^n \to M^n$ be a definable finite-to-one function. Then $f(X)$ has non-empty interior.
\end{Prop}
\begin{proof}
As finite-to-one definable functions preserve dp-rank the proposition follows immediately from Corollary~\ref{cor:interior}.
\end{proof}

\section{dp-minimal valued fields}\label{section:dpminfields}
\textit{Throughout this section $(F,v)$ is a valued field.}
We assume that the reader has some familiarity with valuation theory.
We let $Fv$ the residue field of $(F,v)$, $\mathcal{O}_v$ be the valuation ring and $\mathcal{M}_v$ be the maximal ideal of $\mathcal{O}_v$.
We denote the henselization of $(F,v)$ by $(F^h, v^h)$.
Given a function $p: F^n \rightarrow F^m$ such that $p = (p_1,\ldots, p_m)$ for some $p_1,\ldots, p_m \in F[X_1,\ldots,X_n]$ we let $J_p(a)$ be the Jacobian of $p$ at $a \in F^n$.

\begin{Lem}\label{lem_homeo}
 Let $p\in F[X_1,\ldots,X_n]^n$ and let $B\subseteq (F^h)^n$ be an open polydisc. 
 Suppose that $J_p(a)\neq 0$ for some $a\in (B\cap F^n)$. There is an open polydisc $U\subseteq B$ with $a\in U$ such that the restriction of $p$ to $U$ is injective.\end{Lem}
\begin{proof}
This follows immediately from
\cite[Theorem 7.4]{PZ78}.
\end{proof}

\begin{Prop}\label{prop_pofu}
Suppose that $(F,v)$ is dp-minimal.
 Let $p_1,\ldots,p_n \in F[X_1,\ldots,X_n]$ and $p=(p_1,\ldots,p_n)$.
If $J_p(a)\neq 0$ at $a \in F$ then $p(U)$ has non-empty interior for all nonempty open neighborhoods $U$ of $a$.
\end{Prop}
\begin{proof}
This follows immediately from the previous lemma along with Proposition \ref{prop_intpreservation}.
\end{proof}

The following lemma is included in the proof of the weakly o-minimal case in \cite{MMS} and stated for arbitrary fields in \cite{G14}.
This lemma goes back to \cite{mmv}.

\begin{Lem}\label{lem:ply}
Let $K$ be a field extension of $F$ and let $\alpha \in K \setminus F$ be algebraic over $F$. Let $\alpha = \alpha_1,\ldots,\alpha_n$ be the conjugates of $\alpha$ over $F$ and let $g$ be given by:
\[g(X_0,\ldots,X_{n-1},Y) := \prod_{i=1}^{n}\left (Y - \sum_{j=0}^{n-1} \alpha_i^j X_j \right).\]
Then $g \in F[X_0,\ldots,X_{n - 1},Y]$ and there are $G_0,\ldots,G_{n - 1} \in F[X_0,\ldots,X_{n - 1}]$  such that
\[g(X_0,\ldots,X_{n-1},Y) = \sum_{j=0}^{n-1} G_j(X_0,\ldots,X_{n-1})Y^j + Y^n.\]
Letting $G = (G_0,\ldots, G_{n - 1})$ we have:
\begin{enumerate}
\item If $\bar c = (c_0,\ldots,c_{n - 1})\in F^n$ and $c_j \neq 0$ for some then $g(\bar c,Y)$ has no roots in $F$;
\item There is a $\bar d = (d_0,\ldots, d_{ n - 1})\in F^n$ such that $d_j\neq 0$ for some $j$ and $J_G(\bar d)\neq 0$.
\end{enumerate}
\end{Lem}
In the proposition below $\Gamma$ is the value group of $(F,v)$.
In the ordered case this proposition is a consequence of Proposition 3.6 of \cite{G14} which shows that a dp-minimal ordered field is closed in its real closure.
%Our proof is essentially the same as that given in \cite{G14} and .
\begin{Prop} \label{propC}
Suppose that $(F,v)$ is dp-minimal and let $(F^h, v^h)$ be the henselization of $(F,v)$. Let $\alpha \in F^h$ such that for any $\gamma \in \Gamma$ there is 
some $\beta\in F$ such that $v^h(\beta-\alpha)\geq \gamma$. Then $\alpha\in F$.
\end{Prop}
\begin{proof}
The proof is essentially the same as those of \cite[5.4]{MMS} and
\cite[3.6]{G14}. We give slightly less details.
We suppose that $\alpha$ has degree $n$ over $F$ and let $g$ and $G$ be as in Lemma~\ref{lem:ply}.
Let $\bar d$ be as in (2) above.
By Lemma \ref{lem_homeo}, there is an open set $U\subseteq F^n$ containing $\bar d$ such that the restriction of $G$ to $U$ is injective.
 By Proposition \ref{prop_pofu}, $G(U)$ has non-empty interior. 
As $J_G$ is continuous we may assume, after shrinking $U$ if necessary, that $J_{G}$ is nonzero on $U$.
In the same manner we may suppose that for all $(x_0,\ldots, x_{n - 1}) \in U$ there is a $j$ such that $x_j \neq 0$.
After changing the point $\bar d$ if necessary we may also assume that $\bar e :=  G(\bar d)$ lies in the interior of $G(U)$.
We define a continuous function $f: F^h\setminus \{0\} \to F^h$ by
\[ f(y) := - \left. \left ( y^n + \sum_{j=0}^{n-2} e_{j} y^j \right ) ~\right /~ y^{n-1}.\]
Thus for every $y\neq 0$ we have:
\[ y^n + f(y) y^{n-1} + \sum_{j=0}^{n-2} e_j y^j = 0.\]
We define 
$$h(x) :=\sum_{j=0}^{n-1} d_j x^j.$$
Then $h(\alpha)$ is a zero of $g(\bar d,Y)$, so $h(\alpha) \neq 0$ as $g(\bar d, y)$ has no roots in $F$.
As $h(\alpha)$ is a zero of $g(\bar d, y)$ we also have $f(h(\alpha))=e_{n-1}$. 
If $\beta \in F$ is sufficiently close to $\alpha$ then $h(\beta)\neq 0$ and 
 $$(e_0,\dots,e_{n-2}, (f \circ h)(\beta)) \in V.$$
 There is thus a $\bar c\in U$ with 
 $$G(\bar c) = (e_0,\ldots, e_{n - 2}, (f \circ h)(\beta)).$$
 Now by our choice of $U$, there is a $j$ such that $c_j \neq 0$ and so $g(\bar c,Y)$ has no root in $F$. On the other hand, $h(\beta)$ is a root of $g(\bar c,Y)$. Contradiction.
\end{proof}

\begin{Prop} \label{prop2}
If $(F,v)$ is dp-minimal then $v$ is henselian.
\end{Prop}

\begin{proof}
Suppose that $(F,v)$ is dp-minimal.
Let $\mathcal{O}_v$ be the valuation ring and $\Gamma$ be the value group. 
This proof follows the proofs of \cite[5.12]{MMS} and \cite[3.12]{G14}:
We suppose towards a contradiction that $v$ is not henselian. 
There is a polynomial
$$p(X)=X^n + aX^{n-1} + \sum_{i=0}^{n-2} c_iX^i\in \mathcal{O}_v[X]$$
such that $v(a)=0$, $v(c_i)>0$ for all $i$ and such that $p$ has no root in $F$.
Let $\alpha \in F^h$ be such that $p(\alpha)=0$, $v(\alpha - a) > 0$ and $v(p'(\alpha))=0$.
Consider the subset 
$$S:=\{v^h(b-\alpha)\in \Gamma\,\mid\, b \in F,\,v^h(b-\alpha)>0\}$$
of $\Gamma$, and let $\Delta$ be the convex subgroup of $\Gamma$ generated by $S$.\\
\emph{Claim:} $S$ is cofinal in $\Delta$.\\
\emph{Proof of claim}: Identical with that of \cite[Claim 5.12.1]{MMS}.\\
Let $w$ be the coarsening of $v$ with value group $\Gamma/\Delta$ and let $w^h$ be the corresponding coarsening of $v^h$.
As $\Delta$ is externally definable, $w$ is definable in the Shelah expansion of $(F,v)$.
Then $(F,v,w)$ is dp-minimal and so the residue field $Fw$ of $w$ is also dp-minimal.
We let $\bar{v}$ be the non-trivial valuation induced on $Fw$ by $v$.\\
\emph{Claim:} There is a $\beta \in F$ such that $p(\beta)\in \mathcal{M}_w$.\\
%$(F^hw^h,\bar{v}^h)$ is a henselization of $(Fw,\bar{v})$.\\
\emph{Proof of claim:} 
%I am not sure whether I can find a reference, it's straightfoward anyway. This proves the claim.\\ 
%If $(Fw,\bar{v})$ is henselian then $\alpha \in Fv=Fw\bar{v}$
%lifts to a root of the residue polynomial 
%$\bar{p}(X) \in Kw[X]$ (i.e., the residue of $p(X)$ under $w$) by henselianity.
By the definition of $\Delta$, the residue $\alpha w^h$ is approximated arbitrarily well 
in the residue field $Fw$ (with respect to the valuation $\bar{v}$). 
We show that 
$(F^hw^h,\bar{v}^h)$ is a henselization of $(Fw,\bar{v})$: Since $w$ is
a coarsening of $v$, $(F^h,v^h)$ contains the henselization of $(F,w)$. If 
these henselization coincide, $(Fw, \bar{v})$ is henselian by 
\cite[4.1.4]{EP05}. If the extension is proper, \cite[4.1.4]{EP05} implies once
more 
that $(F^hw^h,\bar{v}^h)$ is a henselization of $(Fw,\bar{v})$, as desired.  
Thus Proposition \ref{propC} gives $\alpha w^h \in Fw$.
Take some $\beta \in F$ with the same residue (with respect to $w$)
as $\alpha$. In particular, $\beta$ is a root of the polynomial $\bar{p}(x)$ (that is $p(x)$ considered in $Kw$),
i.e., we have $p(\beta)\in \mathcal{M}_w$. 
This proves the claim.

We declare $$J:=\{b \in F\,\mid\,v(b-a)>0\}.$$
Then, as $\beta-\alpha \in \mathcal{M}_w \subseteq \mathcal{M}_v$ holds, we have $\beta \in J$.\\
\emph{Claim:} For all $b \in J$, we have $v(b-\alpha)=v(p(b))$.\\
\emph{Proof of claim:} This is shown in the first part of the proof of \cite[Claim 5.12.2]{MMS}.\\
However, by the definition of $\Delta$, $w(p(b))=0$ for any $b \in J$.  This contradicts $p(\beta)\in \mathcal{M}_w$, 
and hence finishes the proof.
\end{proof}

\section{Dp-Minimal Ordered Abelian Groups}\label{section:dpmingroup}
\textit{In this section $(\Gamma,+,\leq)$ is an $\aleph_1$-saturated ordered abelian group with no additional structure}.
In this section we describe dp-minimal ordered abelian groups without additional structure.
Let $M$ be a first order structure expanding a linear order in a language $L$ and suppose that $M$ is $|L|^+$-saturated.
Then $M$ is \textbf{weakly quasi-o-minimal} if every definable subset of $M$ is a boolean combination of convex sets and $\emptyset$-definable sets.
This notion was introduced in \cite{kuda}.
We say that $\Gamma$ is \textbf{non-singular} if $\Gamma/p\Gamma$ is finite for all primes $p$.
Proposition 5.27 of \cite{ADHMS2} implies that a nonsingular torsion free abelian group without additional structure is dp-minimal.

\begin{Prop}\label{dpmingroup}
The following are equivalent:
\begin{enumerate}
\item $(\Gamma, +, \leq)$ is non-singular.
\item $(\Gamma,+,\leq)$ is dp-minimal.
\item There is a definitional expansion of $(\Gamma,+, \leq)$ by countably many formulas which is weakly quasi-o-minimal.
\end{enumerate}
\end{Prop}

\begin{proof}
Theorem 6.8 of \cite{ADHMS1} implies that a weakly quasi-o-minimal structure is dp-minimal so (3) implies (2).
We show that (2) implies (1).
Suppose that $(\Gamma,+,\leq)$ is dp-minimal.
The subgroup $p\Gamma$ is cofinal in $\Gamma$ for any prime $p$. 
A cofinal subgroup of a dp-minimal ordered group has finite index, see \cite[Lemma 3.2]{Simon:dp-min}.
Thus $\Gamma$ is non-singular.

It remains to show that (1) implies (2).
Suppose that $\Gamma$ is non-singular.
We apply the quantifier elimination for ordered abelian groups given in \cite{Immi}.
We use the notation of that paper.
We let $L_{qe}$ be the language described in \cite{Immi}.
Given an abelian group $G$ and $x,y \in G$ we say that $x \equiv_m y$ if $x - y \in mG$.

For $a \in \Gamma$ and prime $p$ we let $H_{a,p}$ be the largest convex subgroup of $\Gamma$ such that $a \notin H_{a,p}+p\Gamma$.
In $L_{qe}$ for each $p$ there is an auxiliary sort $\mathcal S_p = \Gamma/\!\sim$ where $a \sim b$ if and only if $H_{a,p} = H_{b,p}$.
As $H_{a,p}$ only depends on the class of $a$ in $\Gamma/p\Gamma$, $\mathcal S_p$ is finite.
The other auxiliary sorts $\mathcal T_p$ and $\mathcal T_p^+$ parametrize convex subgroups of $\Gamma$ defined as unions or intersections of the $H_{a,p}$, hence they are also finite.
Given an element $\alpha$ of an auxiliary sort we let $\Gamma_\alpha$ be the convex subgroup of $\Gamma$ associated to $\alpha$.
For $k \in \mathbb{Z}$ we let $k_\alpha$ be the $k$th multiple of the minimal positive element of $\Gamma/\Gamma_\alpha$ if $\Gamma/\Gamma_\alpha$ is discrete and set $k_\alpha = 0$ otherwise.
Fix $\alpha$ and let $\pi: \Gamma \rightarrow \Gamma/\Gamma_\alpha$ be the quotient map.
Given $\diamond \in \{=, \leq , \equiv_m\}$ and $a,b \in \Gamma$ we say that $a \diamond_{\alpha} b$ holds if and only if $\pi(a) \diamond \pi(b)$ holds in $\Gamma/\Gamma_\alpha$.
For each $\alpha$ and $m,m' \in \mathbb{N}$ $L_{qe}$ also has a binary relation denoted by $\equiv_{m,\alpha}^{[m']}$.
We do not define this relation here, as for our purposes it suffices to note that the truth value of $a \equiv_{m,\alpha}^{[m']} b$ depends only on the classes of $a$ and $b$ in $\Gamma/m\Gamma$.
As the auxiliary sorts are finite it follows from the main theorem in \cite{Immi} that every definable subset of $\Gamma^k$ is a boolean combination of sets of the form 
$$ \{ \bar x \in \Gamma^k : t(\bar x) \diamond_\alpha t'(\bar x) + k_\alpha\}, $$
for $\mathbb{Z}$-linear functions $t,t'$, $\alpha$ from an auxiliary sort and $\diamond \in \{=,\leq, \equiv_{m}, \equiv_{m}^{[m']}\}$.
If $\diamond$ is $\equiv_{m}^{[m']}$ then $k_\alpha = 0$.
We claim that $(\Gamma, + \leq)$ admits quantifier elimination in the language $L_{short}$ containing:

$\bullet$ the constant 0, the symbols $+$ and $-$ and the order relation $\leq$;

$\bullet$ for each $n$ and each class $\bar a\in \Gamma/n\Gamma$, a unary predicate $U_{n,\bar a}(x)$ naming the preimage of $\bar a$ in $\Gamma$;

$\bullet$ unary predicates naming each subgroup $H_{a,p}$;

$\bullet$ constants naming a countable submodel $\Gamma_0$.

Having named a countable model, we can consider that the auxiliary sorts are in our structure, by identifying each one with a finite set of constants which projects onto it. We do not need to worry about the structure on those sorts since they are finite. Consider a $2$-ary relation $x_1 \diamond_\alpha x_2 + k_\alpha$ where $\diamond \in \{=,<,\equiv_m\}$, $k\in \mathbb Z$, $m\in \mathbb N$, $\alpha$ from an auxiliary sort. If the symbol is equality then this is equivalent to $x_2-x_1+c \in H_{a,n}$ for an appropriate $a,n$ and constant $c \in \Gamma_0$ projecting onto $k_\alpha$. If the symbol is $<$ then we can rewrite it as 
$$[x_1 < x_2 +c] \wedge [x_2-x_1+c \notin H_{a,n}].$$
 Finally, if the symbol is a congruence relation, then its truth value depends only on the images of $x_1$ and $x_2$ in $\Gamma/m\Gamma$, hence the formula is equivalent to a boolean combination of atomic formulas $U_{m,\bar a}(x_1)$ and $U_{m,\bar a}(x_2)$.
The only relations left in $L_{qe}$ are of the form $x\equiv_{m,\alpha}^{[m']} y$, but again their truth value depends only on the image of $x$ and $y$ in $\Gamma/m\Gamma$, so they can be replaced by $L_{short}$ quantifier-free formulas.

Weak quasi-o-minimality follows easily from quantifier elimination in $L_{short}$: an inequality of terms $t(x) \leq t'(x)$ defines a convex set, any atomic formula $t(x) \in H_{a,p}$ for $t$ a term defines a convex set; an atomic formula of the form $U_{m,\bar a}(t(x))$ defines a $\emptyset$-definable set.
\end{proof}

\section{dp-minimal ordered fields}
\textit{In this section $F$ is an ordered field with no additional structure.}
We make use of the following:

\begin{Prop}[\cite{CS:inp-min}]\label{dpminfield}
Let $(K, v)$ be a henselian valued field of equicharacteristic zero.
Then $(K,v)$ is dp-minimal if and only if the residue field and value group of $(K,v)$ are dp-minimal.
\end{Prop}

Given a field $k$ and an ordered abelian group $\Gamma$, $k((t^\Gamma))$ is the field of Hahn series with coefficients in $k$ and exponents in $\Gamma$.
By the Ax-Kochen/Ersov Theorem (\cite[4.6.4]{PD}) a field $K$ 
admitting a henselian valuation with residue characteristic zero, residue field $k$ and value group $\Gamma$ is elementarily equivalent to $k((t^\Gamma))$.

\begin{Thm}\label{orderedf}
The ordered field $F$ is dp-minimal if and only if $F\equiv\mathbb{R}((t^\Gamma))$ for some non-singular ordered abelian group $\Gamma$.
\end{Thm}

\begin{proof}
Suppose that $F\equiv\mathbb{R}((t^\Gamma))$ for a non-singular ordered abelian group $\Gamma$.
Then Proposition~\ref{dpmingroup} and Proposition~\ref{dpminfield} together show that $F$ is dp-minimal.
Suppose that $F$ is dp-minimal.
We suppose without loss of generality that $F$ is $\aleph_1$-saturated.
Let $\mathcal O$ be the convex hull of $\mathbb{Z}$ in $F$.
As $\mathcal O$ is convex, it is externally definable.
Thus $(F, \mathcal O)$ is dp-minimal.
As $\mathcal O$ is a valuation ring, Proposition~\ref{prop2} implies that the associated valuation is henselian.
Let $\Gamma$ be the value group of this valuation.
Then $\Gamma$ is dp-minimal hence non-singular.
The residue field is a subfield of $\mathbb{R}$.
It follows from $\aleph_1$-saturation that the residue field is equal to $\mathbb{R}$.
The Ax-Kochen/Ersov Theorem now implies $F\equiv\mathbb{R}((t^\Gamma))$ (as valued fields).
By \cite[Corollary 3.6]{DelFar}, we also get $F\equiv\mathbb{R}((t^\Gamma))$ (as ordered fields).
\end{proof}

\begin{Cor}
Suppose that $F$ is dp-minimal.
Then $F$ has small absolute Galois group.
In particular, $F$
satisfies the Galois-theoretic assumption in the
Shelah-Hasson Conjecture, i.e.~for all $L \supseteq F$ finite and for all
$n \in \mathbb{N}$, the index $[L^\times:(L^\times)^n]$ is finite.
\end{Cor}
\begin{proof}
Suppose we have $F \equiv \mathbb{R}((t^\Gamma))$ for a nonsingular $\Gamma$.
For each prime $p$ let $r_p$ be the $\mathbb{F}_p$-dimension 
of $\Gamma/p\Gamma$. The proof of Proposition 5.3 in \cite{FJ15a} shows 
that the absolute Galois group of $F$ is isomorphic to the small group
$$\left( \prod_p \mathbb{Z}_p^{r_p} \right) \rtimes \mathbb{Z}/2\mathbb{Z},$$
where the product is taken over all primes $p$.
\end{proof}

\section{dp-minimal fields with type v topologies}\label{section:typeV}
In this section we show that a dp-minimal field which admits a definable type V topology is either real closed, algebraically closed or admits a non-trivial definable henselian valuation.
We first recall the notion of a type V topological field.

Let $F$ be a topological field.
We say that a subset $A \subseteq F$ is \textbf{bounded} if for every 
neighborhood $U$ of zero there is a neighborhood $V$ of zero such that 
$VA \subseteq U$.
It is clear that subsets of bounded sets are bounded and finite unions of bounded sets are bounded, it is not difficult to show that if $A,B \subseteq F$ are bounded then so are $A + B$ and $AB$.
A subset $A \subseteq F$ is \textbf{bounded away from zero} if there 
is a neighborhood of zero which is disjoint from $A$.
We say that $F$ is \textbf{type V} if any set $A \subseteq F$ is bounded if and only if $A^{-1}$ is bounded away from zero.
If $F$ is a type V topological field and $U$ is a bounded neighborhood of zero then the sets of the form $aU + b$ for $a,b \in F$ form a basis for the topology on $F$.
We say that a first order expansion of a field $K$ admits a definable type V topology if there is type V topology on $\mathcal T$ on $K$ and a definable set $U \subseteq K$ which is open and bounded with respect to $\mathcal T$.

\begin{Lem}\label{external}
Let $F$ be a first order structure expanding a field in a language $L$.
Suppose that $F$ is $|L|^+$-saturated and admits a definable type V topology $\mathcal T$.
Then $F$ admits an externally definable valuation $\mathcal O$.
In particular, if $F$ is dp-minimal then $F$ has a valuation ring $\mathcal O$ such that $(F, \mathcal O)$ is dp-minimal.
\end{Lem}

\begin{proof}
Let $F_0$ be an elementary substructure of $F$ of cardinality $|L|$.
Let $\mathcal O$ be the union of all $F_0$-definable bounded sets.
We show that $\mathcal O$ is a valuation ring of $F$.
If $A,B \subseteq F$ are bounded then $A+B$ and $AB$ are also bounded, so $\mathcal O$ is a subring of $F$.
Suppose that $a \notin \mathcal O$.
Then there is a $F_0$-definable open neighborhood $U$ of zero such that $a \notin U$.
As $(F \setminus U)^{-1}$ is bounded we have $a^{-1} \in \mathcal O$.
Thus $\mathcal O$ is a valuation subring of $F$.
Let $\nu$ be the valuation associated to $\mathcal O$.
We show that the $\nu$-topology agrees with $\mathcal T$ on $F$.
If $A \subseteq F$ is bounded and $U \subseteq F$ is bounded and open then $A + U$ is bounded.
Thus every bounded definable subset of $F$ is contained in an open bounded definable subset of $F$.
It follows that $\mathcal O$ is $\mathcal T$-open, so every $\nu$-open set is $\mathcal T$-open.
To show that every $\mathcal T$-open set is $\nu$-open it suffices to fix an $\nu$-open neighborhood $V$ of zero and find an $a \in F$ such that $a \mathcal O \subseteq V$.
Let $V$ be a definable open neighborhood of zero.
For every bounded $F_0$-definable $U \subseteq F$ there is an $a \in F$ such that $aU \subseteq V$.
Applying $|L|^+$-saturation there is an $a \in F$ such that $aU \subseteq V$ holds for every $F_0$-definable bounded set $U$, we have $a \mathcal O \subseteq V$ for this $a$.

We now show that $\mathcal O$ is externally definable.
Let $F'$ be an $|F|^+$-saturated elementary expansion of $F$.
Let $U \subseteq F$ be an $F_0$-definable bounded open neighborhood of zero and let $U'$ be the subset of $F'$ defined by the defining formula of $U$.
For every $F_0$-definable bounded set $V \subseteq F$ there is an $a \in F_0$ such that $V \subseteq aU$.
Furthermore $aU$ is disjoint from $F \setminus \mathcal O$ for every $a \in F_0$.
An application of $|F|^+$-saturation shows that that there is an $a \in F'$ such that $V \subseteq aU'$ holds for every $F_0$-definable bounded set $V$ and $aU'$ is disjoint from $F \setminus \mathcal O$.
Then $ aU' \cap F = \mathcal O$ holds for this $a$.
Thus $\mathcal O$ is externally definable.
\end{proof}

We now show that a dp-minimal field which admits a definable type V topology admits a definable henselian valuation.
We let $L_{ring}$ be the language of rings.
By Lemma~\ref{external} it suffices to show that a dp-minimal valued field admits an $L_{ring}$-definable henselian valuation.
Recall that if $(K,v)$ is an equicharacteristic valued field and $K$ is perfect then the residue field $Kv$ is perfect.
Recall that dp-minimal fields are perfect, see Lemma~\ref{perfect}.
We implicitly use these facts several times in this section to upgrade the assumption `separably closed' to `algebraically closed'.

Let $K$ be a field and $p$ be a prime.
We say that $K$ is \textbf{$p$-closed} if $K$ has no nontrivial Galois extensions of degree $p$.
The term ``$p$-closed'' has been given other meanings in the literature. 
We say that $K$ is \textbf{henselian} if $K$ admits some non-trivial henselian valuation.
A valuation on $K$ is \textbf{$p$-henselian} if it extends uniquely to every Galois extension of $K$ of $p$-power degree.
We say that $K$ is \textbf{$p$-henselian} if $K$ admits a non-trivial $p$-henselian valuation.

A non-separably closed henselian field $K$ admits a
\textbf{canonical
henselian valuation} $v_K$ (cf.\,\cite[p.\,106]{EP05}): 
The valuation rings of henselian valuations on $K$ with non-separably closed residue field are ordered by inclusion.
If $K$ admits no henselian valuation with separably closed residue field then $v_K$ is the finest henselian valuation on $K$.
If $K$ admits a henselian valuation with separably closed residue field then $v_K$
is the coarsest henselian valuation with separably closed residue field. 
In the second case, all henselian valuations with non-separably closed residue field are coarsenings of $v_K$.
In either case, $v_K$ is non-trivial.

Suppose there is a prime $p$ such that $K$ is $p$-henselian and not $p$-closed
(if $p=2$, assume further that $K$ admits a Galois extension of degree 4, i.e., that $K$ is not euclidean). 
Then $K$ admits a \textbf{canonical $p$-henselian valuation} $v_K^p$ on $K$ (cf.\,\cite[p.\,97]{Koe95}).
The valuation rings of $p$-henselian valuations with non $p$-closed residue fields
are ordered by inclusion.
If $K$ does not admit a henselian valuation with $p$-closed residue field then 
$v_K^p$ is the finest $p$-henselian valuation on $K$.
If $K$ admits a henselian valuation with $p$-closed residue field then $v_K^p$
is the coarsest henselian valuation with $p$-closed residue field. 
In this case, all henselian valuations whose residue fields are not $p$-closed are coarsenings of $v_K^p$.
In either case $v_K^p$ is non-trivial.

\begin{Prop} \label{Pdef}
Let $(K,v)$ be a dp-minimal valued field and let $v_K$ be the canonical henselian valuation on $K$.
One of the following holds:
\begin{enumerate}
\item $Kv_K$ is separably closed,
\item $Kv_K$ is  real closed,
\item $v_K$ is $L_{ring}$-definable.
\end{enumerate}
\end{Prop}
\begin{proof}
Let $v_K$ denote the canonical henselian valuation on $K$ and 
assume that $Kv_K$ is neither separably closed nor real closed.
Thus $v_K$ is the finest henselian valuation on $K$.
There is a prime $p$ and a finite Galois extension $k$
of $Kv_K$ such that $k$ admits a Galois extension of degree $p$ (and, if $p=2$, also a Galois extension of degree $4$,
see \cite[Theorem 4.3.5]{EP05}.
There is also a finite Galois extension $L$ of $K$ such that $Lw=k$, where $w$ is the unique prolongation of $v$ to $L$.
If $\mathrm{char}(L)\neq p$ we may additionally assume that $L$ contains a primitive $p$th root
of unity $\zeta_p$: \cite[Theorem 4.3.5]{EP05} implies the residue field of $L(\zeta_p)$ with respect to the (unique) prolongation of $v_K$
still admits a Galois extension of degree $p$.

Note that as $k$ is not $p$-closed, neither is $L$ (\cite[Theorem 4.2.6]{EP05}). 
The canonical $p$-henselian valuation $v_L^p$ on $L$ is nontrivial and 
$L_{ring}$-definable (cf. \cite[Theorem 3.1]{JK14}). 
By the defining properties of $v_L^p$ and as $Lw$ admits a Galois extension of degree $p$, 
$w$ is a coarsening of $v_L^p$. 
As $L$ is interpretable in $K$, the restriction $u$ of $v_L^p$ to $K$ is also a non-trivial
definable valuation on $K$. Since $w$ is a coarsening of $v_L^p$, $u$ is a refinement of $v_L$. 
As $u$ is definable $(K,u)$ is dp-minimal, hence henselian.
 Proposition
\ref{prop2} implies that $(K,u)$ is henselian. 
As $v_K$ has no henselian refinements we conclude that $u=v_K$.
Thus $v_K$ is definable.
\end{proof}

\begin{Cor}\label{ringdef}
 Let $(K,v)$ be a dp-minimal valued field. Then either $K$ is algebraically closed or real closed or $K$
admits a non-trivial $L_{ring}$-definable henselian valuation.
\end{Cor}
\begin{proof}
Assume that $K$ is neither algebraically closed nor real closed.
It follows that $K$ is not separably closed.
Let $v_K$ denote the canonical henselian valuation on $K$.
 If $Kv_K$ is separably closed (respectively real closed), 
then $K$ admits a non-trivial definable henselian valuation by \cite[Theorem 3.10]{JK14a} (respectively \cite[Corollary 3.11]{JK14a}).
Otherwise, $v_K$ is definable by Proposition \ref{Pdef}. 
\end{proof}

By Lemma~\ref{external} and Corollary~\ref{ringdef} we have:

\begin{Prop}
Let $K$ be a dp-minimal field admitting a definable type V topology.
Then $K$ is either real closed, algebraically closed or admits a non-trivial definable henselian valuation.
\end{Prop}

\textit{For the remainder of this section $F$ is an ordered field with real closure $R$}.
To end this section, we give an explicit construction of an 
$L_\textrm{of}$-definable non-trivial 
valuation on a dp-minimal ordered field which is not real closed. 
Here $L_\textrm{of}$ denotes
the language of ordered fields, i.e., $L_\textrm{of}=L_\textrm{ring}\cup\{\leq\}$.
By Proposition \ref{prop2}, any such valuation is 
henselian.
The existence of such a valuation is already implied by the
results above, however, the proofs given above 
are non-constructive.
In fact, we prove a stronger result: an ordered field is either dense in its real closure or admits a definable valuation.
%The following proposition states that a dp-minimal field is closed in its real closure:
%
%\begin{Prop}[{\cite[3.6]{G14}}] \label{propG}
%Fix $\alpha \in R$, and suppose that for each $\epsilon \in R$
%with $\epsilon > 0$, there exists $b \in F$ such that $|\alpha-b|<\epsilon$. Then, $\alpha \in F$.
%\end{Prop}
%
%We require the following version:

\begin{Lem} \label{corG}
Let $F$ be an ordered field with real closure $R$. Fix $\alpha \in R$, and suppose that for each $\epsilon \in F$
with $\epsilon > 0$, there exists $b \in F$ such that $|\alpha-b|<\epsilon$. Then, $\alpha$ is in the closure of $F$.
\end{Lem}
\begin{proof}
%By Proposition \ref{propG},
We need to show that for each $\epsilon \in R_{>0}$ there is some $\delta \in F_{>0}$ with $\delta<\epsilon$.
Assume not, so there is some $\epsilon  \in R_{>0}$ such that for all 
$\delta \in F_{>0}$ the inequality $\epsilon<\delta$ holds.
Then, if $f(X) = \sum_{i\leq n} a_i X^ i\in F[X]$ is irreducible we have 
$a_0\neq 0$ and $-\delta<\sum_{0<i\leq n} a_i \epsilon^ i<\delta$ holds for all 
$\delta \in F_{>0}$. Thus, $f(\epsilon)\neq 0$. This contradicts the fact that $R$ is algebraic
over $F$.
\end{proof}

We can now give the final results of this paper:
\begin{Prop}\label{prop1} Let $F$ be an ordered field with real closure $R$.
One of the following holds:
\begin{enumerate}
\item $F$ is dense in $R$;
\item $F$ admits an $L_\textrm{of}$-definable valuation.
\end{enumerate}
\end{Prop}
\begin{proof}
We suppose that $F$ is not dense in $R$.
Let $\alpha$ be an element of $R$ which is not in the closure of $F$.
Without loss of generality, we may assume $\alpha>0$.
Then
$$D:= \{a \in F \,\mid\, a < \alpha\}$$ 
is a definable subset of $F$. Hence,
$$A:=\{y \in F_{\geq 0}\,\mid\,y+D \subseteq D\}$$
is also a definable subset of $F$.\\
\emph{Claim:} $\{0\} \subsetneq A \subsetneq F_{\geq 0}$ is a proper convex semigroup of $F$.\\
\emph{Proof of claim:} It is easy to see that $A$ is a convex and closed under addition. Furthermore, $A$ is closed under addition
since for $y_1,y_2 \in A$ and $d \in D$ we have 
$$(y_1+y_2)+d = y_1 + \underbrace{(y_2+d)}_{\in D} \in D.$$
Finally, assume $A=\{0\}$. Then, for all $\epsilon \in F_{>0}$ there is some $d \in D$ with $d+\epsilon > \alpha$,
hence $\alpha - d < \epsilon$. Lemma \ref{corG} now implies that $\alpha$ is in the closure of $F$, a contradiction.
This proves the claim.\\
Let $\mathcal{O} = \{ a \in F : aA \subseteq A \}$ be the multiplicative stabilizer of $A$.
It is easy to see that $\mathcal{O}$ is a convex subring of $F$.
%As $A$ is an additive subgroup, we get $\mathbb{Z} \subseteq \mathcal{O}$, 
For any $b \in F_{\geq 0}$ with $b > a$ for all $a \in A$ we have $b \notin 
\mathcal{O}$. Thus $\mathcal{O}$ is non-trivial.
\end{proof}

\begin{Cor}
Let $F$ be a dp-minimal ordered field which is not real closed. Then, the definable valuation constructed in the proof
of Proposition \ref{prop1} is non-trivial and henselian.
\end{Cor}
\begin{proof} Assume that $F$ is a dp-minimal ordered field. By \cite[Proposition 3.6]{G14},
$F$ is closed in its real closure $R$. Thus, if $F$ is not real closed, $F$ cannot be dense in $R$
and thus admits a definable non-trivial valuation ring $\mathcal{O}$ by Proposition \ref{prop1}.
Now, Proposition \ref{prop2} implies that $\mathcal{O}$ is henselian.
\end{proof}

\bibliographystyle{alpha}
\bibliography{franzi}
\end{document}